\documentclass[preprint,12pt]{elsarticle}
\usepackage{comment}

\usepackage{amssymb}
\usepackage{comment}
\usepackage{amsmath}
\usepackage{amsmath, tikz}
\usetikzlibrary{fit, positioning}
\usepackage[european]{circuitikz}

 \newtheorem{theo}{Theorem}
 \newtheorem{lem}{Lemma}
 \newdefinition{rmk}{Remark}
 \newdefinition{defi}{Definition}
 \newdefinition{assum}{Assumption}
 \newproof{proof}{Proof}
 \newproof{pot}{Proof of Theorem \ref{thm2}}

\journal{Computers \& Mathematics with Applications}

\begin{document}

\begin{frontmatter}



\title{Pathwise convergence of a linearization scheme for  stochastic differential-algebraic  equations under the local  Lipschitz coefficients}


\author[gt]{Guy Tsafack} 
\ead{guy.tsafack@hvl.no (guytsafack4@gmail.com)}
\address[gt]{Department of Computer science, Electrical engineering and Mathematical sciences,  Western Norway University of Applied Sciences, Inndalsveien 28, 5063 Bergen, Norway.}
\author[gt,at1]{Antoine Tambue} 
\ead{antoine.tambue@hvl.no (antonio@aims.ac.za)}
\address[at1]{The Department of Mathematics \& Applied Mathematics, University of Cape Town, Private Bag, 7701
Rondebosch, Cape Town, South Africa }
\begin{abstract}
The paper deals with the numerical treatment of index-1  stochastic differential-algebraic equations (SDAEs) with nonlinear coefficients that satisfy the local Lipschitz and the Khasminskii conditions. The key challenge  here is the presence of a singular and non-autonomous matrix in the equation, which makes the numerical method challenging to  analyze. To tackle this challenge, we develop a more general numerical method using a local linearization technique. More precisely, we use the Taylor expansion to decompose  locally the drift component of the SDAEs in linear and nonlinear parts. The linear part is approximated implicitly and must resolve the singularity issue of each time step, while the  nonlinear part is approximated explicitly.
  This method is fascinating due to the fact that it
  is efficient  in high dimension.
  We prove that this novel numerical method converges in the pathwise sense with rate  $\frac{1}{2}-\epsilon$, for arbitrary $\epsilon >0$. 
The implementation of this novel numerical method is also carried out to verify our theoretical result. 

\end{abstract}




\end{frontmatter}



\section*{Introduction}
We are interesting in the numerical solution of the following  general stochastic differential algebraic equation:
\begin{equation}\label{equa1}
	A(t)dZ(t)=f(t,Z(t))dt + g(t,Z(t))dW(t), \quad t\in \left[0,T\right] ,~Z(0)=\zeta\in D,
\end{equation}
where $T>0,\quad T\ne \infty$, and  matrix  A : $  t \longmapsto A(t) $ is continuously differentiable and   $ A(t)\in  \mathcal{M}_{d\times d}(\mathbb{R})$ is a singular matrix for all  $ t\in \left[0,T \right]$.
The coefficients  $f:\left[0,T\right]\times D  \to \mathbb{R}^d  $ and $g:\left[0,T\right]\times D\to \mathbb{R}^{d\times d_1} $ are 
Borel functions on $\left[0,T\right]\times D$ with $D\subseteq\mathbb{R}^d$ . 
The process $W(\cdot)$ is a $d_1$-dimensional  Wiener process defined in the probability space $(\Omega, \mathcal{F},\mathbb{P})$ with natural filtration $\left( \mathcal{F}_t\right)_{t\geq 0} $. 
The unknown variable $Z$ is a  $d-$dimensional stochastic process  that depends on the time $t\in \left[0, T\right] $ and the sample $\omega\in\Omega$. To simplify the notation, the argument $\omega$ is omitted. Our assumption on equation \eqref{equa1} will allow the solution $Z(t)$ to never leave $D$ and just for convenience, we assume that $f(t,Z(t))=0_d$\footnote{null vector in $\mathbb{R}^d$} and $g(t,Z(t))=0_{d\times d_1}$\footnote{null matrix in $\mathbb{R}^{d\times d_1}$} for $Z\notin D,~t\in [0,T]$.\\


Stochastic modeling has become essential across numerous scientific disciplines \cite{pavliotis2014stochastic,guttorp2018stochastic} and industrial applications \cite{grigoriu2013stochastic, henderson2006stochastic}, as many real-world phenomena inherently involve noise or random fluctuations. 
Stochastic equations exist in various forms, with stochastic ordinary differential equations being among the most extensively studied \cite{Khasminskii,mao2015truncated,Mao2007}.
However, when modeling real world problems, it is recommended to incorporate some conditions to obtain a more realistic model \cite{campbell2019applications}.
These additional conditions are called constraint equations. Combining the two equations (Stochastic Differential Equations (SDEs) and constraints), we obtain  the stochastic differential algebraic equations (SDAEs) \cite{schein1999stochastic, denk2008efficient, tsafack2025strong}.  In simple terms, SDAEs are the general form of SDEs.
An important example where SDAEs usually arise is in modeling electronic circuits disturbed by the so-called electrical noise \cite{schein1998numerical, penski1999analysis, Renate}. In the literature, the theoretical and numerical studies of the SDEs and Differential Algebraic Equations (DAEs) are well developed and have grown considerably over the last 40 years \cite{hairer2006numerical, wanner1996solving,klebaner2012introduction,lawler2018introduction, Mao2007, mao2015truncated}. However, the study of SDAEs in both theoretical and numerical treatment is under development.  The first results in SDAEs were the study of the existence and uniqueness of the solution under the global Lipschitz condition \cite{Renate, cong2010stochastic,schein1998numerical}. However, most of the equations do not satisfy the global Lipschitz condition \cite{mao2015truncated, izgi2021strong}. Recently  in \cite{serea2025existence, monarticle2}, the existence and uniqueness of the solution of SDAEs under non-global Lipschitz conditions has been studied. Due to the non-linearity nature of such equations, it is very difficult or sometimes impossible to find the exact solution, even if the well-posedness is well understood.  
 So, numerical studies  are the only hope for realistic applications.
The first numerical method was developed in \cite{schein1999stochastic}, where the authors proposed a two-step method as well as a generalized Euler method for linear SDAEs. Note that this method is restricted to linear SDAEs with multi-dimensional additive noise. In  \cite{romisch2003stochastic, Renate, kupper2012runge}, the authors introduced numerical methods for index-1 nonlinear SDAEs with globally Lipschitz continuous coefficients. In the first approach, they proposed a drift-implicit Euler method, a split-step backward Euler scheme, and a trapezoidal method \cite{romisch2003stochastic, Renate}. In the second approach, they used a Runge–Kutta method  and studied the mean-square error \cite{ kupper2012runge}.
For all those stable  implicit time stepping methods, the keys challenge is that at each iteration, we need to solve systems of nonlinear algebraic equations with, for example, Newton's method,
which may lead to bottlenecks in practical computations in high dimensions.


Note that  several authors have studied pathwise convergence in the context of stochastic differential equations \cite{kloeden2007pathwise, jentzen2009pathwise,nguyen2012pathwise}. This involves computing the pathwise error, which is a random quantity but provides more information about the error of the approximated solution. It is important to note that in several problems with less regularity \cite{kloeden2007pathwise, jentzen2009pathwise}, it is challenging to study the strong or the weak convergence, and then the only alternative method is the study of the pathwise convergence.  
It arises naturally in many important applications \cite{kloeden2007pathwise, jentzen2009pathwise}. 
In \cite{monarticle2}, pathwise convergence is studied for equation of type \eqref{equa1} with semi linear function $f$, this means that $f$ can be written as $  f(t,Z)=B(t)Z+h(t,Z)$, where $B(t)$ is a matrix independent of $Z$ and $h(., .)$ is a nonlinear function, non-globally Lipschiz respect to $Z$. The scheme developed in \cite{monarticle2} was based on semi implicit method where the matrix  $B(t)$  has for duty to solve the singularity issue of $A(t)$  at each iteration. 
The scheme in \cite{monarticle2} has been efficient to solve large system of SDAEs with semi linear function $f$.  However, in the case where the drift part $f$ is fully  nonlinear ($B(t)=0$), the method used in \cite{monarticle2} cannot be applied. To the best of our knowledge, pathwise convergence for general SDAEs is not yet well understood in the literature. 

The objective of this paper is to develop and analyze a novel numerical method  to approximate the solutions of nonlinear, non-autonomous stochastic differential-algebraic equations of type \eqref{equa1} under non-global Lipschitz and Khasminskii-type conditions and to establish its pathwise convergence. To construct the proposed numerical method, we begin by linearizing the drift component of the SDAEs using the Taylor approximation. Subsequently, a semi-implicit discretization scheme is applied to the transformed system, wherein the implicit part is explicitly designed to address the singularity issues inherent in the original SDAEs. This leads to the construction of a nonsingular matrix defined by $D_n=A(t_n) -\frac{T}{N}J(t_n,Z_n)$ where $J$ denotes the Jacobian matrix \footnote{The first-order partial derivative with respect to the variable $Z$} of $f$,  for all $N\in \mathbb{N}$, and $n=0,1,2,...,N-1$. 
The main idea of the method is the local linearization of the original stochastic differential-algebraic equation, resulting  to a linear system that is solved  at each time step to move forward. Due to the fact that the linearization is local, the corresponding scheme is extremely much difficult to analyze comparing to  scheme developed in  \cite{monarticle2} for semi linear SDAEs. Under the assumption that the inverse of each matrix $D_n$ is uniformly bounded, we provide a rigorous mathematical analysis  of  pathwise convergence of the proposed method.


This paper is organized as follows. Section \ref{sec1} introduces the necessary assumptions and discusses the well-posedness of the stochastic differential-algebraic equations of type \eqref{equa1} under non-global Lipschitz and Khasminskii-type conditions. In Section \ref{sec2}, we demonstrate that the standard explicit Euler–Maruyama method is not suitable for approximating solutions of SDAEs, unlike in the case of SDEs. This limitation arises because of  the presence of a singular matrix and the nonlinearity of the SDAEs coefficients. In addiction, we propose our  novel numerical method to approximate  the samples of  SDAEs of type  \eqref{equa1} and provide its pathwise convergence analysis. Finally, Section \ref{sec3} presents numerical simulations to illustrate our theoretical  result.

\section{Well-posedness of the SDAEs}\label{sec1}

Throughout this work,  $(\Omega, \mathcal{F},\mathbb{P})$ denotes a complete probability space with natural
filtration $\left( \mathcal{F}_t\right)_{t\geq 0} $. We define $\left\|z \right\|^2=\sum_{i=1}^{d}\left| z_i\right|^2$ for any vector $z\in \mathbb{R}^d$ and $\left| B \right|^2_F=\sum_{i=1}^{d}\sum_{j=1}^{d_1}\left|b_{i,j}\right| ^2 $ is the Frobenius norm for any matrix $B=(b_{i,j})_{i,j=1}^{d, d_1}$.  We also denote by $\left\|\cdot\right\|_{\infty}$ the supremum norm for continuous functions. Moreover, for $A:t\mapsto  A(t),~t\in \left[ 0,T\right]$, we have $\left\| A\right\|_{\infty}:=\max_{t\in \left[ 0,T\right]}\left| A(t)\right| _F$. 

\begin{defi}\label{def1}
				The SDAEs (\ref{equa1}) is said to be of index-1  if the noise sources do not appear in the constraints and the constraints (AEs) are globally uniquely solvable. 
			\end{defi}
         Note that  the noise sources do
not appear in the constraints means that $Img(t,Z)\subseteq ImA(t)$, for all $Z\in D, t\in \left[0,T\right] $.
We continue with the following hypothesis, which is necessary for the well-posedness result. 
\begin{assum}\label{asum1} We assume that

\begin{enumerate}

\item [(A1.1)]  The functions $f(\cdot,\cdot)$ and $g(\cdot,\cdot)$ are continuous and locally Lipschitz with respect to $Z$ i.e. for any $q \geq 1$,   there exists $L_q > 0$, such that
\begin{equation}\label{equa4}
		\left\| f(t,Z)-f(t,Y)\right\|\lor	\left| g(t,Z)-g(t,Y)\right|_F \leq L_q \left\|Y- Z\right\|,~t\in \left[0,T\right],~~~~~~~~~
\end{equation}
for all $ Z, Y\in D$ with $ \left\|Z \right\|\lor \left\|Y \right\|<q $.
\item [(A1.2)]   There exists an increasing sequence of bounded domains $\left\lbrace D_q\right\rbrace_{q=1}^{\infty} $ such that\\ $\cup_{q=1}^{\infty}D_q=D$,  and for every $q\geq 1$, $t\in \left[0,q\right],$
				$$\sup_{Z\in D_q}\left\|f(t,Z) \right\|\leq M_q, ~~\sup_{Z\in D_q}\left|g(t,Z) \right|_F^2\leq M_q ,~~~\text{where}~ M_q~ \text{is a constant.}~~~~~~~$$ 
                Note that similar assumptions can be found in \cite{gyongy1998note, gyongy1996existence, monarticle2}.
				\item[(A1.3)] The functions $f(\cdot,\cdot)$ and $g(\cdot,\cdot)$ satisfied the Khasminskii condition \footnote{ See for example \cite[ Theorem 3.5]{Khasminskii} or in \cite{gyongy1998note}.}, i.e., there exists a non-negative function $V(\cdot, \cdot)\in \mathcal{C}^{1,2}(\left[0,T \right]\times D)$ such that
				$$LV(t,Z)\leq MV(t,Z),  \forall t\in \left[0,T \right], Z\in D,~~~~~~~~~~~~~~~~~~~~~~~~~~~~~~~~~~~~~~~~~~~~~~~~$$ 
				$$ V_q(T):=\inf_{ X\in \partial D_q, ~~ t\leq T} V(t,Z)\to \infty, ~~ as ~~ q\to \infty , \forall T<\infty;~~~~~~~~~~~~~~~~~~~~~~~~~~~~$$
				where $M=M(T)$  is a constant, $\partial D_q$ denotes the boundary of $D_q$ and $L$ is the differential operator 
				$$L=\frac{\partial}{\partial t}+\sum_{i}(A^-f)_i(t,Z)\frac{\partial}{\partial z_i}+\frac{1}{2}\sum_{i}(A^-g(A^-g)^T)_{ij}(t,Z)\frac{\partial^2}{\partial z_i \partial z_j}.~~~~~~~~~~~~~~~~~~~~~~~~~~~~~~~~~~~~~$$
				\item [(A1.4)]  $\mathbb{P}(\zeta\in D)=1$.
\item [(A1.5)]  The singular matrices $A(t),~t \in \left[0,T \right] $ and the pseudo-inverse matrix $A^-(t) , ~ t\in \left[0,T \right]$ are such that the function $P(t)=A^-(t)A(t)$ is differentiable and satisfy the relation
\begin{equation}\label{eqp2}
	(A^-(t)A(t))'=0_{d\times d},~~t\in \left[ 0,T\right].~~~~~~~~~~~~~~~~~~~~~~~~~~~~~~~~~~~~~~~~~~~~~~~~~~
\end{equation}
Here, $A'(t)$ represents the derivative with respect to the time $t$, and $0_{d\times d}$ is a null matrix of size $d\times d$.

\end{enumerate}
\end{assum}
\begin{rmk}\label{rem34}
	  Note that if $D=\mathbb{R}^d$, we can take  $D_q=\{Z\in \mathbb{R}^d, ~~\|Z\|<q,\;\;\;q\geq 1\},$  and  Assumption \ref{asum1}  (A1.2)  becomes
		$$ \sup_{\|Z\|<q}\{\|f(t,Z)\|+|g(t,Z)|_F^2\}\leq 2M_q, \text{ for every  } t\in [0,T] \text{ and } q\geq 1.$$
		In addition by  taking  $V(t,Z)=(1+\|Z\|^2)e^{-Mt} ;~M>0$,  the Assumption \ref{asum1} (A1.3)  gives   the  standard monotonicity  condition and for all $ t\in [0,T], ~Z\in \mathbb{R}^d $ 
		$$\langle (A^{-}(t)A(t) Z)^T, A^-(t)f(t,Z)\rangle+ \frac{1}{2}|A^-(t)g(t,Z)|_F^2\leq M(1+\|Z\|^2).$$
        Note that  $A^-=A^-AA^-$. 
	\end{rmk}
		\begin{rmk}\label{rem1}
		Note that a function  $f$ defined on $[0,T]\times D$ and satisfying (A1.2) of  Assumption \ref{asum1}   is locally Lipschitz in $ D$  if there exists  a bounded measurable   function  $f_q: [0,T]\times \mathbb{R}^d\to \mathbb{	R}^d$ such that $f_q(t,Z)=f(t,Z)$  for  $ t\in [0,T],~Z\in D_q,\,\, q\geq 1$ and 
				\begin{eqnarray}
				\left\{\begin{array}{l}
				\|f_q(t,Z)\|	\leq L_q, ~Z\in \mathbb{R}^d , t\in \left[0,T \right],\\
					\newline\\
						\|f_q(t,Z)-f_q(t,Y)\|\leq L_q\|Z-Y\|, ~Z,Y\in \mathbb{R}^d , t\in \left[0,T \right].
				\end{array}\right.
			\end{eqnarray} 
		\end{rmk}	
        The next lemma plays a crucial role in our analysis.
        \begin{lem}\label{lemaw}
	Let $Z(t)$ be  the $\mathcal{F}_t$-adapted process  solution  of \eqref{equa1}  for $t<\tau:=\inf\{t: Y(t)\notin D\}$.  Assume that the Assumption \ref{asum1}  is  satisfied,  then $\tau=\infty$ a.s.
\end{lem}
The proof of this lemma is similar to that of \cite[Lemma 3]{monarticle2}. 

			We denote by  $J_{AE}(t,Z)=A(t)+R(t)f_Z(t,Z)$,   $t\in \left[0,T\right] $ and $Z\in D$, the Jacobian matrix of the constraint equation with $R(t)$ the projector matrix associated to $A(t)$,  $t\in \left[0,T \right]$. See \cite{Renate, serea2025existence,monarticle2} for more explanations. 
	
			\begin{theo}\label{theo1}
				Assume  that the equation \eqref{equa1}  is an index 1 SDAEs and the Jacobian matrix $J_{AE}(\cdot,\cdot)$   possesses a global bounded  inverse for any $Z\in D$. Suppose also that Assumption \ref{asum1} is satisfied.
				Then a unique process $Z(\cdot)$  solution of \eqref{equa1} exists. 
			\end{theo}

			\begin{proof}
				For this proof, we refer the reader to the proof of \cite[Theorem 1]{monarticle2}.
			\end{proof}
            As mentioned in the introduction, it is generally difficult or even impossible to find the exact solution of SDAEs. Our novel numerical scheme is provided in the next section.

			\section{ Novel numerical scheme based on local linearization}\label{sec2}

            This section first discusses the challenges associated with using the explicit Euler Maruyama method to solve stochastic differential-algebraic equations, which arise due to the singularity of the matrix $A(\cdot)$  and the nonlinearity of the coefficients. In the second part, we develop a new semi-implicit linearization scheme to sample the solutions of  SDAEs of type \eqref{equa1}.
			
            
            The following assumptions are important for the main result in this section.
            \begin{assum}\label{assum2}   For a  given positive number $h>0$, we assume that
			\begin{enumerate}
            \item [(A2.1)] The  function $f$ is  twice continuously  differentiable and  the matrices $(A(t)-J(t,Z)h),~(A(t)+R(t)J(t,Z)) ,~Z\in \mathbb{R}^d, t\in \left[ 0,T\right]$ are  non singular matrices, where $R$ is a projector matrix associated to the matrix $A$ and $J$ is the Jacobian matrix of the function $f$ with respect to  $Z$. 
			\item[(A2.2)] The matrices $S^h(t,Z))=(A(t)-J(t,Z))h)^{-1}$, $(A(t)+R(t)J(t,Z))^{-1}$ and $J(t,Z),~Z\in \mathbb{R}^d, t\in \left[ 0,T\right]$ are bounded and Lipschitz with the variable $t$ by the constant $k_1$,
			\item [(A2.3)] There exists a constant $k_2$ such that for  $t_n, ~t_m\leq T$, we have  $$\left| \frac{\partial f}{\partial t}(t_n,Z(t_n))\right| \leq k_2, ~\left| \frac{\partial f}{\partial t}(t_n,Z(t_n))- \frac{\partial f}{\partial t}(t_m,Z(t_m))\right|\leq k_2\left|t_n-t_m\right|. $$
            Note that this assumption is not too  restrictive and was also used in \cite{jimenez2017weak}.
             \item  [(A2.4)] 	There exists a constant $k_3$ such that for  $t_n\leq T$,  we have $$\left\| D^2f(t_n,X_n)\right\|\leq  k_3,$$
            where $D^2f:=(f_{tt}, f_{tx}, f_{xt},f_{xx})$.
				\end{enumerate}
				 \end{assum}
               Throughout this work, for $N\in \mathbb{N}$, we denote by  $h=\frac{T}{N}$  the time step-size, $\Delta W_{n}=W(t_{n+1})-W(t_n)$, $t_{n+1}=t_n+ h=(n+1)h $ and $X_{n}\approx Z(t_{n}) ,~n=0,1,2,3,..., N-1$, with   $t_0=0$ and $t_N=T$.
			\subsection{Euler Maruyama method}
			 Applying the  Euler Maruyama scheme to the SDAEs \eqref{equa1} \footnote{
			as for SDEs in \cite{mao2015truncated,gyongy1998note,Renate,kelly2022adaptive}} yields
			\begin{equation}\label{equa6a}
				A(t_n)X_{n+1}=A(t_n)X_{n}+f(t_n,X_{n})h+g(t_n,X_n)\Delta W_{n}, ~X_0=\zeta.
			\end{equation}
            In addition to the stability issue from the explicit nature  of the scheme \eqref{equa6a}, the singularity issue of the matrix $A(t_n)$ makes the scheme heavy to implement as the constraints need to be eliminated.
			Our novel scheme based on semi-implicit and local linearization technique  presented in the next section,
            will be more competitive in terms of stability and efficiency, while approximating samples of the SDAEs \eqref{equa1}.
          \subsection{ Semi-implicit linearization scheme for the SDAEs}
            Basically, our novel technique based on local linearization involves the local decomposition of the drift term into linear and nonlinear components using Taylor's expansion. 
          Indeed we consider the  first-order  Taylor approximation of  $f(\cdot,\cdot)$
			\begin{equation}\label{equa36q}
				f(t,x)\approx f(t_n,x_n)+\frac{\partial f}{\partial x}(t_n,x_n)(x-x_n)+\frac{\partial f}{\partial t}(t_n,x_n)(t-t_n),
			\end{equation}
             $ t_n \leq t \leq t_{n+1}, \, x\in D$.
             Using the approximation  \eqref{equa36q}  in  \eqref{equa1} and following \cite{monarticle2}  yields for $n=0,1,2,...,N-1$ 
			\begin{align}\label{equa36a}
				A(t_n^N)(X_{n+1}-X_n)&=\left[ J(t_n^N,X_n)X_{n+1}+F(t_n^N, X_n)+\frac{\partial f}{\partial t}(t_n^N,X_n)h\right]h~~~~~~~~~~~~~~~~~~~\nonumber\\
                &+g(t_n^N,X_n)\Delta W_{n},
			\end{align}
			with
			$$J(t_n^N,X_n)=\dfrac{\partial f}{\partial X} (t_n^N,X_n )=:J_n,  \,\,\,\,F(t_n^N,X_n)=f(t_n^N,X_n)-J_nX_n.$$
             The scheme \eqref{equa36a} can be rewritten as
			\begin{align}\label{equa37a}
				(A(t_n^N)-J_nh)X_{n+1}&=A(t_n^N)X_{n} +F(t_n^N, X_n)h+\frac{\partial f}{\partial t}(t_n^N,X_n)h^2+g(t_n^N,X_n)\Delta W_{n}.
			\end{align}
			 By setting $S^h(t_n^N,X_n)=S^h_n=(A(t_n^N)-J_nh)^{-1},$ the scheme can be rewritten as 
			\begin{align}\label{equa38a}
				X_{n+1}&=S^h_nA(t_n^N)X_{n} +S^h_nF(t_n^N, X_n)h+S^h_n\frac{\partial f}{\partial t}(t_n^N,X_n)h^2
				+S^h_ng(t_n^N,X_n)\Delta W_{n}.
			\end{align}
			The corresponding  continuous-time approximation $X^N(t),~ t\in [t_n^N; t_{n+1}^N]$  with $X^N(t_n^N)=X_n$ is given by
			\begin{align}\label{equa39a}
				X^{N}(t)&=S^h_nA(t_n^N)X^{N}(t_n^N) +S^h_nF(t_n^N, X^N(t_n^N))(t-t_n^N)+S^h_n\frac{\partial f}{\partial t}(t_n^N,X^N(t_n^N))h(t-t_n^N)~~~~~~~~~~~~~~\nonumber\\
				&+S^h_ng(t_n^N,X^N(t_n^N))(W(t)-W(t_n^N)).
			\end{align}
We can also   be rewritten  \eqref{equa39a} as follows
 \begin{align*}
				X^{N}(t)&=S^h(k_N(t_n^N))A(k_N(t_n^N))X^{n}(k_N(t_n^N))\\ &+S^h(k_N(t_n^N))F(k_N(t_n^N), X^N(k_N(t_n^N)))(t-k_N(t_n^N))\\
                &+S^h(k_N(t_n^N))\frac{\partial f}{\partial t}(k_N(t_n^N),X^N(k_N(t_n^N)))h(t-k_N(t_n^N))\\
				&+S^h(k_N(t_n^N))g(k_N(t_n^N),X^N(k_N(t_n^N))(W(t)-W(t_n^N)),~t\in \left[ t_n^N, t_{n+1}^N\right] ,
			\end{align*} 
   with $k_N(t)=t_n^N=\frac{nT}{N},~ t\in \left[ t_n^N ,t_{n+1}^N\right)$. \\
   
   Since the $X^N(t)$ is only defined in $\left[ t_n^N ,t_{n+1}^N\right]$. For greater rigor, we define the continuous-time approximation on the interval $[0,T]$ as follows
   $$ X_N(t)=X^N(t),~ \text{if } t\in \left[ t_n^N ,t_{n+1}^N\right].$$
   
			In our analysis, it will be more natural to work with the following representation form \footnote{See   \cite{gyongy1996existence,higham2002strong, gyongy1998note, monarticle2} for more information.}
			\begin{align}\label{equa41}
				X_{N}(t)&=S^h_0A(t_0)\zeta +\int_{0}^{t}S^h(k_N(s))F(k_N(s), X_N(k_N(s)))ds~~~~~~~~~~~~~~~~~~~~~~~~~~\nonumber\\
               &+\int_{0}^{t}S^h(k_N(s))\frac{\partial f}{\partial t}(k_N(s), X_N(k_N(s)))hds\nonumber\\
              &+\int_{0}^{t}S^h(k_N(s))g(k_N(s),X_N(k_N(s)))dW(s),~~~t\in \left[ 0, T\right] ,
			\end{align}		
			where $k_N(t)=t_n^N=\frac{nT}{N},~ t\in \left[ t_n^N ,t_{n+1}^N\right)$. \\
			
			The following result is our main result.
			\begin{theo}\label{theo3} 
				Assume that Assumptions \ref{asum1} and \ref{assum2} hold. Then the approximation solution $X_N(t), ~~N\in \mathbb{	N}$   converges  in probability to the  exact solution $Z(t)$ of \eqref{equa1} almost surely,  uniformly in $t\in [0,T]$. Moreover,  for every $\alpha <\frac{1}{2}$ and $T>0$  there exists  a finite random variable $\epsilon$, such that:
				\begin{equation}\label{equa8a}
					\sup_{t\leq T}\left\| X_N(t)-Z(t)\right\|\leq \epsilon N^{-\alpha} .
				\end{equation} 
			\end{theo}
            The proof of this theorem is technically  heavy and therefore requires several intermediate results.
            \subsubsection{Preliminary results for Theorem \ref{theo3}}
Note that SDAEs can always be decomposed into a stochastic differential equations (SDEs) and a system of algebraic constraint equations (AEs). This implies that the solution $X$ of the SDAEs can be written as $X=u+\hat{v}(u)$, where $u$ is the solution of the associated SDEs and $\hat{v}$ denotes the implicit function solution of the algebraic equations. \\

The central preliminary result  is  the  following equivalence result.
\begin{lem}\label{lem3} Assume that the matrix $A(t_n^N)+R(t_n^N)J_n,n\leq N\in\mathbb{N}$ is non singular.  
	Then the scheme defined in equation \eqref{equa36a} is equivalent to the following scheme
	\begin{equation}\label{equa5bf}
		\left \{
		\begin{array}{c c c}
			u_{n+1}-u_n &=&P'(t_n^N)\left[ u_{n+1}+\hat{v}_N(t_{n+1}^N,u_{n+1})\right] h~~~~~~~~~~~~~~~~~~~~~~~~~~~~~~~~~~~~~~~~~~\\
            \\
            &&+ A^-(t_n^N) J_n\left[u_{n+1}+\hat{v}_N(t_{n+1}^N,u_{n+1})\right]h~~~~~~~~~~~~~~~~~~~~~~~~~~~~~~~~~~~~\\
			\newline\\
			&&+A^-(t_n^N)(F(t_n^N,u_{n}+\hat{v}_N(t_{n}^N,u_{n}))h~~~~~~~~~~~~~~~~~~~~~~~~~~~~~~~~~~~~~~ \\
            \\
            &&+A^-(t_n^N)\frac{\partial f}{\partial t}(t_n^N, u_{n}+\hat{v}_N(t_{n}^N,u_{n}))h^2)~~~~~~~~~~~~~~~~~~~~~~~~~~~~~~~~~~~~~~\\
            \newline\\
            &&+A^-(t_n^N)g(t_n^N,u_{n}+\hat{v}_N(t_{n}^N,u_{n}))\Delta W_n,~~~~~~~~~~~~~~~~~~~~~~~~~~~~~~~~~~~~~\\
			\newline\\
           \hat{v}_N(t_{n+1}^N,u_{n+1})&=&-(A(t_n^N)+R(t_n^N)J_n)^{-1}\left[R(t_n^N)J_nu_{n+1}\right.~~~~~~~~~~~~~~~~~~~~~~~~~~~~~~~~~~~\\
           \newline\\
           &&\left.+R(t_n^N)(F(t_n^N,u_{n}+\hat{v}_N(t_{n}^N,u_{n}))\right.~~~~~~~~~~~~~~~~~~~~~~~~~~~~~~~~~~~~~~~~~~\\
           &&\left. + \frac{\partial f}{\partial t}(t_n^N, u_{n}+\hat{v}_N(t_{n}^N,u_{n}))h)\right],~~~~~~~~~~~~~~~~~~~~~~~~~~~~~~~~~~~~~~~~~~~~\\
                      \newline\\
              u_0&=&P(0)\zeta,~~~~~~~~~~~~~~~~~~~~~~~~~~~~~~~~~~~~~~~~~~~~~~~~~~~~~~~~~~~~~~~~~~~~~~~~~\\
              \hat{v}_N(0,u_0)&=&Q(0)\zeta ,~~~~~~~~~~~~~~~~~~~~~~~~~~~~~~~~~~~~~~~~~~~~~~~~~~~~~~~~~~~~~~~~~~~~~~~~~\\
               \newline\\
			 X_{n+1}&=&u_{n+1} +\hat{v}_N(t_{n+1}^N,u_{n+1}),	~ n=0,1,...,N-1,~N\in\mathbb{N},~~~~~~~~~~~~~~~~~~~
		\end{array}\right. 
	\end{equation}
	where the  matrix $A^{-}(t)$ is the pseudo-inverse matrix of the matrix $A(t)$ and,   $P(t)$  and $R(t)$ are projectors matrices associated to $A(t)$,  $t\in \left[0,T \right] $ ( See \cite{serea2025existence} for more information about the matrix R).
\end{lem}
\begin{proof}
The continuous form of the equation \eqref{equa36a} can be defined as
\begin{align}\label{sdae}
 A(k_N(t))dX_N(t)&=[J_n(k_N(t))X_N(t)+F(k_N(t),X_N(k_N(t))]dt\nonumber\\
 &+\frac{\partial f}{\partial t}(k_N(t),X_N(k_N(t)))hdt\nonumber\\
 &+g(k_N(t),X_N(k_N(t)))dW(t)   ,~~
 X_N(0)=\zeta; t\in [0,T].
\end{align}
From \cite[Lemma 5]{monarticle2} there exist the  projector matrices $Q(t)$, $R(t)$ and $P(t)$ such that $ImQ(t)=KerA(t),  ~ R(t)A(t)=0_{d\times d}, \text{ and } Q(t)=I-P(t),~ t\in [0,T]$ and then we have this projection
\begin{eqnarray}
	\left \{
	\begin{array}{l}
		X_N(t)=P(t)X_N(t)+Q(t)X_N(t)=u_N(t)+v_N(t),\\
		\newline\\
		 u_N(t)\in \text{Im}P(t), \quad v_N(t)\in \text{Im}Q(t) ,\quad t\in\left[0,T \right].
	\end{array}\right. 
       \end{eqnarray}
      The remainder of the proof follows a similar approach to that used in \cite[Lemma 7]{monarticle2}.
\end{proof}
\begin{lem}\label{lem2}
	Let Assumptions \ref{asum1} and \ref{assum2} be satisfied.  Consider the solution $X_N^q(.)$ of equation \eqref{equa36a} or \eqref{equa41}, where the coefficients $f$ and $g$ are replaced by the functions $f_q$ and $g_q$ respectively, as defined in the Remark \ref{rem1}. Then the following inequality holds
	$$\mathbb{	E}\left\|  X_N^q(s)\right\|^{p}\leq C, ~s\in \left[0, T\right] , ~p\geq 1, ~C\text{ is a constant independent to N.} ~~~~~~~~~~~~~~~~~~~~~~~~~~~~~~~~~~~~~~$$
\end{lem}
          \begin{proof}	
          We use the projection defined in the previous lemma, which means that $X_N^q(t)=u_N^q(t)+\hat{v}_N(t, X_N^q(t))$ and we have
          \begin{align*}
              \mathbb{	E}\left\|X_N^q(t)\right\|^{p}&=\mathbb{	E}\left\|u_N^q(t)+\hat{v}_N(t, u_N^q(t))\right\|^{p}\\
              &\leq \mathbb{	E}\left\|u_N^q(t)+\hat{v}_N(t,0)+\hat{v}_N(t, u_N^q(t))-\hat{v}_N(t,0)\right\|^{p}\\
               &\leq \mathbb{	E}(\left\|u_N^q(t)\right\|+\left\|\hat{v}_N(t,0)\right\|+\left\|\hat{v}_N(t, u_N^q(t))-\hat{v}_N(t,0)\right\|)^{p}\\
               &\leq 2^{p-1}((1+L_{\hat{v}})^p\mathbb{	E}\left\|u_N^q(t)\right\|^p+\left\|\hat{v}_N(t,0)\right\|^p).
          \end{align*}
				The remaining steps of the proof are similar to those in \cite[Lemma 8]{monarticle2}.
			\end{proof}

                  The next lemma is a regularity result.      
				\begin{lem}\label{theo3a}
				Let $(X_N^q(t))_{t\in \left[ 0,T \right] }$ be a solution of $\eqref{equa41}$  and $Z^q(t)$ the solution of equation (\ref{equa1}) where the coefficients $f(\cdot,\cdot)$ and $g(\cdot,\cdot)$ are replaced by $f_q(\cdot,\cdot)$ and $g_q(\cdot,\cdot)$ defined in Remark \ref{rem1}. Then
				for all $p\geq 1$, there exists a  positive constant  $\hat{M}$ such that
				\begin{align}\label{equa39}
					\mathbb{E}\left[\left\|X_N^q(t)-X^q_N(t_n) \right\|^p \right]\leq 					\hat{M}\left| t-t_n\right|^{p/2} , t\in [t_n,t_{n+1}),\text{ and }\nonumber\\
                    \nonumber\\
                    \mathbb{E}\left[\left\|Z^q(t)-Z^q(t_n) \right\|^p \right]\leq 					\hat{M}\left| t-t_n\right|^{p/2} , t\in [t_n,t_{n+1}),
				\end{align}
                with $n\leq N\in \mathbb{N}$.
			\end{lem}
			
			\begin{proof}	
				We can use a similar approach as in \cite[Lemma 9]{monarticle2}.
			\end{proof}
			\begin{lem}\label{lemma4}
         Let $\hat{v}$ the implicit function solution of the constraint equation from the SDAEs \eqref{equa1} and $\hat{v}_N$ its numerical approximation defined in Lemma \ref{lem3}.  Assume that Assumptions \ref{asum1} and \ref{assum2} are satisfied. Then the two time depending functions $u^q_N$ and $u^q$ defined on  $D=\cup_{q=1}^{\infty}D_q$, satisfy the following inequality
  \begin{align*}
       \|\hat{v}_N(t,	u^q_N(t))- \hat{v}(t,u^q(t))\|^2
    & \leq C\left(h^2+\left\|u^q_N(t)-u^q(t)\right\|^2\right.\nonumber\\
                       &\left.+\| u^q(t)-u^q(t_n)\|^2\right)  + C\left\|u^q_N(t)-u^q_N(k_N(t))\right\|^2\nonumber\\
                    &+C\left\|X^q_N(t)-Z^q(t)\right\|^2+Ch^2 \left\|X^q_N(k_N(t))\right\|^2\nonumber\\
                    &+Ch^2 \left\|u^q_N(t)\right\|^2.
    \end{align*} 
    C is the constant independent to $n\leq N  \in \mathbb{N}$ and $h=\frac{T}{N}$.
            \end{lem}
            \begin{proof}	
                
                
                Following a similar approach as in the proof of \cite[Lemma 7, Eq. (22)]{monarticle2}, we obtain the constraint equation associated with the numerical SDAEs \eqref{sdae} where the coefficients are replaced by $f_q, g_q$ defined from Remark \ref{rem1}. We therefore  have 
               \begin{align}\label{17u}
                A(k_N(t))v_N(t)&+R(k_N(t))J^q_N[u_N(t)+v_N(t)]+R(k_N(t))f_q(k_n(t),X_N(k_N(t)))\nonumber\\
                &-R(k_N(t))J^q_N[X_N(k_N(t))]+R(k_N(t))\frac{\partial f_q}{\partial t}h=0.
                \end{align}
               
                This means that the implicit function $\hat{v}_N(.,.)$ is a solution of the following equation
                \begin{align}\label{17u}
            \hat{v}_N(t,u^q_N(t))=&-(  A(k_N(t)+R(k_N(t))J^q_N(k_N(t)))^{-1}R(k_N(t))J^q_Nu^q_N(t) \nonumber\\
                &+(  A(k_N(t)+R(k_N(t))J^q_N(k_N(t)))^{-1}R(k_N(t))J^q_N X^q_N(k_N(t))\nonumber\\
                &-(  A(k_N(t)+R(k_N(t))J^q_N)^{-1}R(k_N(t))f_q(k_N(t),X^q_N(k_N(t)))\nonumber\\
                &-(  A(k_N(t)+R(k_N(t))J_N)^{-1}R(k_N(t))\frac{\partial f_q}{\partial t}h.
                \end{align}
                 Following a similar approach as in the proof of \cite[Theorem 1, Eq. (10)]{serea2025existence}, we obtain the constraint equation associated with the initial  SDAEs \eqref{equa1} 
                 \begin{align*}
                     A(t)v(t)+R(t)f_q(t, Z(t))=0.
                 \end{align*}
Using  Taylor expansion on the function $f_q$ at $(t_n, ~x_n)$, we obtain

                   \begin{align}\label{19u}
                     A(t)v(t)+R(t)\left[f_q(t_n,x^q_n)+J^q(t_n)(Z^q-x^q_n)+\frac{\partial f_q}{\partial t}(t_n,x^q_n)h+ R^q_1(t_n,h)\right]=0,
                 \end{align}
                 where $$R^q_1(t_n,h)=\int_0^1(1-s)\left(f^q_{tt}(\gamma (s))h^2+2f^q_{tz}(\gamma (s))[h,\Delta Z]+f^q_{zz}(\gamma (s))[\Delta Z, \Delta Z]\right)ds,$$
                 and $\gamma (s)=(t_n+sh, Z(t_n+sh));~ \Delta Z=Z^q(t_n+sh)-Z^q(t_n);~s\in [0,1]$.\\
                 
                 This means that  the implicit function $\hat{v}(t,u^q(t))$ is solution of the equation \eqref{19u} with
                \begin{align}\label{20u}
                    \hat{ v}(t, u^q(t))&=-(A(t)+R(t)J^q(t_n))^{-1}R(t)J^q(t_n)u^q(t)\nonumber\\
                    &+(A(t)+R(t)J^q(t_n))^{-1}R(t)J^q(t_n)Z^q(t_n)\nonumber\\
                    &-(A(t)+R(t)J^q(t_n))^{-1}R(t)f_q(t_n,Z^q(t_n))\nonumber\\
                    &-(A(t)+R(t)J^q(t_n))^{-1}R(t)[\frac{\partial f_q}{\partial t}(t_n,Z^q(t_n))h+R_1^q(t_n,h)].
                 \end{align}
          Note that the implicit functions $\hat{v}^N$ and $\hat{v}$ are globally Lipschitz with respect to the second variable  $u^n$ or $u$. Indeed  we can just prove that $\left\|\frac{\partial \hat{v}^n}{\partial u}(t,u)\right\|\leq L_{\hat{v}}$ \footnote{See for example \cite[Part 2 in the proof of Theorem 1]{serea2025existence} for more explanation.},  where $L_{\hat{v}}$ is also the Lipschitz constant independent to $n$. 
                
                 By taking \eqref{20u}-\eqref{17u}  we obtain
                 \begin{align*}
                     \hat{ v}(t, u^q(t))-&   \hat{v}_N(t,u^q_N(t))=-(A(t)+R(t)J^q(t_n))^{-1}R(t)J^q(t_n)u^q(t)\\
                    &+(  A(k_N(t)+R(k_N(t))J^q_N(k_N(t)))^{-1}R(k_N(t))J^q_Nu^q_N(t) \\
                    &+(A(t)+R(t)J^q(t_n))^{-1}R(t)J^q(t_n)Z^q(t_n)\\
                &-(  A(k_N(t)+R(k_N(t))J^q_N(k_N(t)))^{-1}R(k_N(t))J^q_N X^q_N(k_N(t))\\
                &-(A(t)+R(t)J^q(t_n))^{-1}R(t)f_q(t_n,Z^q(t_n))\\
                &+(  A(k_N(t)+R(k_N(t))J^q_N)^{-1}R(k_N(t))f_q(k_N(t),X^q_N(k_N(t)))\\
                &-(A(t)+R(t)J^q(t_n))^{-1}[R(t)\frac{\partial f_q}{\partial t}(t_n,Z^q(t_n))h+R_1^q(t_n,h)]\\
                &+(  A(k_N(t)+R(k_N(t))J^q_N)^{-1}R(k_N(t))\frac{\partial f_q}{\partial t}h.
                 \end{align*}
                 By taking the norm square on both sides, we have
                  \begin{align}\label{21u}
                      & \left\| \hat{ v}(t, u^q(t))-   \hat{v}_N(t,u^q_N(t))\right\|^2\leq 5\left(\left\|-(A(t)+R(t)J^q(t_n))^{-1}R(t)J^q(t_n)u^q(t)\right.\right.\nonumber\\
                    &~~~~~~~~~~~~~~~~~~~~~~~\left.\left.+(  A(k_N(t)+R(k_N(t))J^q_N(k_N(t)))^{-1}R(k_N(t))J^q_Nu^q_N(t)\right\|^2\right.\nonumber\\
                    &~~~~~~~~~~~~~~~~~~~~~~~+\left.\left\|(  (A(t)+R(t)J^q(t_n))^{-1}R(t)J^q(t_n)Z^q(t_n)\right.\right.\nonumber \\
                &~~~~~~~~~~~~~~~~~~~~~~~\left.\left.-(  A(k_N(t)+R(k_N(t))J^q_N(k_N(t)))^{-1}R(k_N(t))J^q_N X^q_N(k_N(t))\right\|^2\right.\nonumber\\
                &~~~~~~~~~~~~~~~~~~~~~~~+\left.\left\|-(A(t)+R(t)J^q(t_n))^{-1}R(t_n)f_q(t_n,Z^q(t_n))\right.\right.\nonumber\\
                &~~~~~~~~~~~~~~~~~~~~~~~\left.\left. +(  A(k_N(t)+R(k_N(t))J^q_N)^{-1}R(k_N(t))f_q(k_n(t),X^q_N(k_N(t)))\right\|^2\right.\nonumber\\
                &~~~~~~~~~~~~~~~~~~~~~~~+\left.\left\|-(A(t)+R(t)J^q(t_n))^{-1}[R(t)\frac{\partial f_q}{\partial t}(t_n,Z^q(t_n))h+R_1^q(t_n,h)]\right.\right.\nonumber\\
                &~~~~~~~~~~~~~~~~~~~~~~~\left.\left. +(  A(k_N(t)+R(k_N(t))J^q_N)^{-1}R(k_N(t))\frac{\partial f_q}{\partial t}h\right\|^2\right)\nonumber\\
                &~~~~~~~~~~~~~~~~~~~~~~~:=5(A_1+A_2+A_3+A_4+A_5),
                 \end{align} 
                 with:
                 \begin{align*}
                   A_1&:=  \left\|-(A(t)+R(t)J^q(t_n))^{-1}R(t)J^q(t_n)u^q(t)\right.\\
                    &\left.+(  A(k_N(t)+R(k_N(t))J^q_N(k_N(t)))^{-1}R(k_N(t))J^q_Nu^q_N(t)\right\|^2.\\
                    \\
                    A_2&:=\left\|(  (A(t)+R(t)J^q(t_n))^{-1}R(t)J^q(t_n)Z^q(t_n)\right. \\
                &\left.-(  A(k_N(t)+R(k_N(t))J^q_N(k_N(t)))^{-1}R(k_N(t))J^q_N X^q_N(k_N(t))\right\|^2.\\
                \\
                A_3&=\left\|-(A(t)+R(t)J^q(t_n))^{-1}R(t_n)f_q(t_n,Z^q(t_n))\right.\\
                &\left. +(  A(k_N(t)+R(k_N(t))J^q_N)^{-1}R(k_N(t))f_q(k_N(t),X^q_N(k_N(t)))\right\|^2.\\
                \\ A_4&=\left\|-(A(t)+R(t)J^q(t_n))^{-1}R(t)\frac{\partial f_q}{\partial t}(t_n,Z^q(t_n))h\right.\\
                &\left. +(  A(k_N(t)+R(k_N(t))J^q_N)^{-1}R(k_N(t))\frac{\partial f_q}{\partial t}(k_N(t),X^q_N(k_N(t)))h\right\|^2.\\
                A_5&=\left\|(A(t)+R(t)J^q(t_n))^{-1}R_1^q(t_n,h)]\right\|^2.
                        \end{align*}
                 Let us estimate each term. We use the following equation in our estimation
                \begin{equation}\label{22u}
            ab-cd=a(b-d)+(a-c)d.
                \end{equation}
                  This means that for the estimation of $A_1$ we have
                 \begin{align*}
                      A_1&:= \left\|-(A(t)+R(t)J^q(t_n))^{-1}R(t)J^q(t_n)u^q(t)\right.\\
                    &\left.+(  A(k_N(t)+R(k_N(t))J^q_N(k_N(t)))^{-1}R(k_N(t))J^q_Nu^q_N(t)\right\|^2\\
                &\leq \left\|(  A(t)+R(t)J^q(t_n))^{-1}R(t)J^q(t_n)\right\|^2\left\|u^q(t)-u^q_N(t)\right\|^2+ \left\|u^q_N(t)\right\|^2 \\
                &\times\left\| A(k_N(t)+R(k_N(t))J^q_N(k_N(t)))^{-1}R(k_N(t))J^q_N\right.\\
                &\left.-(A(t)+R(t)J^q(t_n))^{-1}R(t)J^q(t_n) \right\|^2\\
                &\leq C(k_1)(\left\|u^q_N(t)-u^q(t)\right\|^2+\left\|u^q_N(t)\right\|^2\left\|t-k_N(t)\right|^2).
                 \end{align*}
                    Using the same method as in $A_1$ we obtain
                    \begin{align*}
                      A_2    &\leq C(k_1)(\left\|X^q_N(k_N(t))-Z^q(t)\right\|^2+\left\|X^q_N(k_N(t))\right\|^2\left|t-k_N(t)\right|^2)\\
                      &\leq C(k_1)\left(\left\|X^q_N(k_N(t))-X^q_N(t)\right\|^2+\left\|X^q_N(t)-Z^q(t)\right\|^2\right.\\
                      &\left.+\left\|X^q_N(k_N(t))\right\|^2\left|t-k_N(t)\right|^2\right).
                 \end{align*}

                   For  the estimation of  $A_3$ , we also have
               \begin{align*}
                   A_3&:=  \left\|(  A(k_N(t)+R(k_N(t))J^q_N)^{-1}R(k_N(t))f_q(k_N(t),X^q_N(k_N(t)))\right.\\
                    &\left.-(A(t)+R(t)J^q(t_n))^{-1}R(t_n)f_q(t_n, Z^q(t_n))\right\|^2\\
                    &\leq  \left\|(  A(k_N(t)+R(k_N(t))J_N)^{-1}R(k_N(t))\right\|^2 \\
                    &\times \left\|f_q(k_N(t),X^q_N(k_N(t)))-f_q(t_n, Z^q(t_n))\right\|^2+\left\|f_q(t_n, Z^q(t_n))\right\|^2\\
                    &\times \left\|(  A(k_N(t)+R(k_N(t))J^q_N)^{-1}R(k_N(t))-(A(t)+R(t)J^q(t_n))^{-1}R(t)\right\|^2 .
                 \end{align*}
                 From Assumption \ref{asum1}, Assumption \ref{assum2} and Remark \ref{rem1} we have
                                \begin{align*}
                   A_3
                    &\leq  C(L_q,k_1) \left\|X^q_N(k_N(t))-Z^q(t_n)\right\|^2\\
                    &+  C(L_q,k_1) \left\|(  A(k_N(t)+R(k_N(t))J^q_N)^{-1}R(k_N(t))-(A(t)+R(t)J^q(t_n))^{-1}R(t)\right\|^2 \\
                    &\leq  C(L_q,k_1)\left\|X^q_N(k_N(t))-Z^q(t_n)\right\|^2+ C(L_q,k_1)\left|t-k_N(t)\right|^2.
                 \end{align*}
               
               
                 Let us estimate $A_4$ also
                 \begin{align*}
                      A_4&:=h^2\left\|-(A(t)+R(t)J^q(t_n))^{-1}R(t)\frac{\partial f_q}{\partial t}(t_n,Z^q(t_n))\right.\\
                &\left. +(  A(k_N(t)+R(k_N(t))J^q_N)^{-1}R(k_N(t))\frac{\partial f_q}{\partial t}(k_n(t),X^q_N(k_N(t)))\right\|^2\\
                &\leq h^2(\left\|(A(t)+R(t)J^q(t_n))^{-1}R(t)\right\|^2\left\|\frac{\partial f_q}{\partial t}(t_n,Z^q(t_n))\right\|^2\\
                & +\left\|(  A(k_N(t)+R(k_N(t))J^q_N)^{-1}R(k_N(t))\right\|^2\left\|\frac{\partial f_q}{\partial t}(k_n(t),X^q_N(k_N(t)))\right\|^2\\
                      &\leq C(L_q,k_1,k_2)h^2.
                 \end{align*}
                 Let us find the estimation of $A_5$
                 \begin{align*}
                  A_5&=\left\|(A(t)+R(t)J^q(t_n))^{-1}R_1^q(t_n,h)]\right\|^2\\
                  &\leq k_1\left\|R_1^q(t_n,h)]\right\|^2\\
                  &\leq k_1\int_0^1\left\|(1-s)\left(f^q_{tt}(\gamma (s))h^2+2f^q_{tz}(\gamma (s))[h,\Delta Z]+f^q_{zz}(\gamma (s))[\Delta Z, \Delta Z]\right)\right\|^2ds\\
                 & \leq k_1k_3^2(h^2+2\left\|(h,\Delta Z)\right\|^2+\left\|(\Delta Z, \Delta Z)\right\|^2\int_0^1\left(1-s\right)^2ds\\
                 &\leq C(h^2+\|\Delta Z\|^2)=C(h^2+\| Z(t_n+sh)-Z(t_n)\|^2)\\
                 &\leq C(h^2-\| Z(t_n+sh)-Z(t_n)\|^2).
                 \end{align*}
                 By replacing $A_1$, $A_2$, $A_3$ , $A_4$ and $A_5$ in \eqref{21u} and using the fact that $\left|k_N(t)-t\right|\leq h$ we obtain
                   \begin{align*}
                       &\left\| \hat{ v}(t, u^q(t))-   \hat{v}_N(t,u^q_N(t))\right\|^2   \leq C\left(h^2+\left\|X^q_N(k_N(t))-X^q_N(t)\right\|^2\right.\\
                       &~~~~~~~~~~~~~~~~~~~~\left.+\| Z^q(t_n+sh)-Z^q(t_n)\|^2\right)\\
                    & ~~~~~~~~~~~~~~~~~~~~+ C \left\|X^q_N(k_N(t))-Z^q(t_n)\right\|^2+C\left\|X^q_N(t)-Z^q(t)\right\|^2\nonumber\\
                    &~~~~~~~~~~~~~~~~~~~~+C\left\|u^q_N(t)-u^q(t)\right\|^2+Ch^2( \left\|X^q_N(k_N(t))\right\|^2+\left\|u^q_N(t)\right\|^2).
                       \end{align*}
                    This means that
                     \begin{align}\label{23u}
                       &\left\| \hat{ v}(t, u^q(t))-   \hat{v}_N(t,u^q_N(t))\right\|^2   \leq C\left(h^2+\left\|u^q_N(t)-u^q(t)\right\|^2\right.\nonumber\\
                       &~~~~~~~~~~~~~~~~~~~~\left.+\| u^q(t)-u^q(t_n)\|^2\right)  + C\left\|u^q_N(t)-u^q_N(k_N(t))\right\|^2\nonumber\\
                    &~~~~~~~~~~~~~~~~~~~~+C\left\|X^q_N(t)-Z^q(t)\right\|^2+Ch^2( \left\|X^q_N(k_N(t))\right\|^2+\left\|u^q_N(t)\right\|^2).
                       \end{align}
              This ends the proof our lemma.         
			\end{proof}
            The following and last lemma establishes the convergence in probability of the numerical scheme.
           \begin{lem} \label{exis}
     Assume that Assumptions \ref{asum1} and \ref{assum2} hold, then the numerical scheme \eqref{equa41}  defined on $D_q,~q\geq1$   converges in probability,  uniformly on the interval $[0, T ]$ to  the unique solution $Z_q$ of the SDAEs \eqref{equa1}, as $N\to \infty$. 
     \end{lem}
\begin{proof}
Using the fact that the function $f$  is local Lipschitz in $D$ then for every $q\geq1$ there is a bounded measurable function $f_q$ such that the functions $f_q$ and $f$ agree on $[0,T]\times D_q$ (see Remark \ref{rem1}). The same applies to the functions  $g_q$ and $g$.
We note by $X_N^q(t), N\in\mathbb{	N},q>0$ the approximation solution of equation \eqref{equa41} with the coefficient $f_q$ and $g_q$.\\
				 From scheme \eqref{equa41}, we have for $N,~M\in \mathbb{N}$
				\begin{align}\label{equa42}
					X^q_{N}(t)&=S^h_0A(t_0)\zeta +\int_{0}^{t}S^h(k_N(s))F_q(k_N(s), X_N^q(k_N(s)))ds~~~~~~~~~~~~~~~~~~~~~~~~~~\nonumber\\
                             &+\int_{0}^{t}S^h(k_N(s))\frac{\partial f_q}{\partial t}(k_N(s), X_N^q(k_N(s)))hds\nonumber\\
                             &+\int_{0}^{t}S^h(k_N(s))g_q(k_N(s),X_N^q(k_N(s)))dW(s),
				\end{align}	
                and 
				\begin{align}\label{equa43}
					X^q_{M}(t)&=S^h_0A(t_0)\zeta +\int_{0}^{t}S^h(k_M(s))F_q(k_M(s), X_M^q(k_M(s)))ds~~~~~~~~~~~~~~~~~~~~~~~~~~\nonumber\\
                            &+\int_{0}^{t}S^h(k_M(s))\frac{\partial f_q}{\partial t}(k_M(s), X_M^q(k_M(s)))hds\nonumber\\
                            &+\int_{0}^{t}S^h(k_M(s))g_q(k_M(s),X_M^q(k_M(s)))dW(s).
				\end{align}	
				Taking the difference \eqref{equa42}-\eqref{equa43}, we obtain the following result
				\begin{align}\label{equa44}
					X^q_{N}(t)-X^q_{M}(t)&=\int_{0}^{t}S^h(k_N(s))F_q(k_N(s), X_N^q(k_N(s)))ds\nonumber\\
                                      &- \int_{0}^{t}S^h(k_M(s))F_q(k_M(s), X_M^q(k_M(s)))ds\nonumber\\
                                      &+h\int_{0}^{t}S^h(k_N(s))\frac{\partial f_q}{\partial t}(k_N(s), X_N^q(k_N(s)))ds\nonumber\\
                                      &-h\int_{0}^{t}S^h(k_M(s))\frac{\partial f_q}{\partial t}(k_M(s), X_M^q(k_M(s)))ds\nonumber\\
                                      &+\int_{0}^{t}S^h(k_N(s))g_q(k_N(s),X_N^q(k_N(s)))dW(s)\nonumber\\
                                      &-\int_{0}^{t}S^h(k_M(s))g_q(k_M(s),X_M^q(k_M(s)))dW(s).
				\end{align}	
    
				We apply the It\^o formula \cite[Theorem 6.2]{Mao2007} on the function $\left\|	X^q_{N}(t)-X^q_{M}(t) \right\|^2 $,  for $t\in \left[0,T \right]. $
				\begin{align}\label{equa45}
					&\left\|	X^q_{N}(t)-X^q_{M}(t) \right\|^2=2\int_{0}^{t} \left\langle X_N^q(s)-X_M^q(s),  S^h(k_N(s))F_q(k_N(s), X_N^q(k_N(s)))\right\rangle ds\nonumber\\
					&~~~~~~~~~~~~~~~~~~-2\int_{0}^{t} \left\langle X_N^q(s)-X_M^q(s),S^h(k_M(s)) F_q(k_M(s), X_M^q(k_M(s)))\right\rangle ds\nonumber\\
                    &~~~~~~~~~~~~~~~~~~+2h\int_{0}^{t} \left\langle X_N^q(s)-X_M^q(s),S^h(k_N(s))\frac{\partial f_q}{\partial s}(k_N(s), X_N^q(k_N(s)))\right\rangle ds\nonumber\\
                    &~~~~~~~~~~~~~~~~~~-2h\int_{0}^{t} \left\langle X_N^q(s)-X_M^q(s),S^h(k_M(s))\frac{\partial f_q}{\partial s}(k_M(s), X_M^q(k_M(s)))\right\rangle ds\nonumber\\
					&~~~~~~~~~~~~~~~~~~+\int_{0}^{t}\left|S^h(k_N(s))g_q(k_N(s),X_N^q(k_N(s)))\right.\nonumber\\
					&~~~~~~~~~~~~~~~~~~ \left. -S^h(k_M(s))g_q(k_M(s),X_M^q(k_M(s))) \right|_F^2ds\nonumber\\
                    &~~~~~~~~~~~~~~~~~~+2\int_{0}^{t}\left\langle  X_N^q(s)-X_M^q(s), S^h(k_N(s))g_q(k_N(s),X_N^q(k_N(s)))\right\rangle dW(s)  \nonumber\\
                   & ~~~~~~~~~~~~~~~~~~-2\int_{0}^{t}\left\langle  X_N^q(s)-X_M^q(s), S^h(k_M(s))g_q(k_M(s),X_M^q(k_M(s))) \right\rangle dW(s).
				\end{align}
				To simplify the presentation, we define $a1(s), ~b1(s),~c1(s),~d1(s)$ by
				\begin{align*}
				    a1(s)&=2\left\langle X_N^q(s)-X_M^q(s),  S^h(k_N(s))F_q(k_N(s), X_N^q(k_N(s)))\right\rangle \\
        &-2\left\langle X_N^q(s)-X_M^q(s),S^h(k_M(s))F_q(k_M(s), X_M^q(k_M(s)))\right\rangle ,
        \end{align*}
        \begin{align*}
        b1(s)&=2h\left\langle X_N^q(s)-X_M^q(s),S^h(k_N(s))\frac{\partial f_q}{\partial s}(k_N(s), X_N^q(k_N(s)))\right\rangle \\
        &-2h\left\langle X_N^q(s)-X_M^q(s),S^h(k_M(s))\frac{\partial f_q}{\partial s}(k_M(s), X_M^q(k_M(s)))\right\rangle ,
        \end{align*}
        \begin{align*}
        c1(s)&=2\left\langle  X_N^q(s)-X_M^q(s),S^h(k_N(s))g_q(k_N(s),X_N^q(k_N(s))) \right\rangle\\
        &-2\left\langle  X_N^q(s)-X_M^q(s),S^h(k_M(s))g_q(k_M(s),X_M^q(k_M(s)))\right\rangle,
        \end{align*}
        \begin{align*}
        d1(s)&=\left|S^h(k_N(s))g_q(k_N(s),X_N^q(k_N(s)))-S^h(k_M(s))g_q(k_M(s),X_M^q(k_M(s))) \right|_F^2.\\
				\end{align*}
                Following the same process as in \cite[eq (41) ]{monarticle2}, we can establish 
that  for $p\geq 2$,  \eqref{equa45} becomes
			\begin{align}\label{equa46}
				&\mathbb{	E}\sup_{s\leq t}\left(\exp(-p\left\|\zeta \right\|) \left\|	X^q_{N}(s)-X^q_{M}(s) \right\|^{2p} \right) \leq \nonumber\\
                    &~~~~~~~~~~~~~~~~~~~~~~~~~~~~~~4^{p-1}T^{\frac{p}{2}}	 \mathbb{	E}\left\lbrace \int_{0}^{t}\exp(-2\left\|\zeta \right\| )\left|a1(s) \right| ^2ds\right\rbrace^{\frac{p}{2}}\nonumber\\
				&~~~~~~~~~~~~~~~~~~~~~~~~~~~~~~+ 4^{p-1}T^{\frac{p}{2}}	\mathbb{	E}\left\lbrace \int_{0}^{t}\exp(-2\left\|\zeta \right\| )\left|b1(s) \right| ^2ds\right\rbrace^{\frac{p}{2}}\nonumber\\
					&~~~~~~~~~~~~~~~~~~~~~~~~~~~~~~ + 4^{p-1}T^{\frac{p}{2}}\mathbb{	E}\left\lbrace \int_{0}^{t}\exp(-2\left\|\zeta \right\| )\left|d1(s) \right| ^2ds\right\rbrace^{\frac{p}{2}}\nonumber\\
     &~~~~~~~~~~~~~~~~~~~~~~~~~~~~~~+ C(p,T)\mathbb{	E}\left\lbrace  \int_{0}^{t}\exp(-2\left\|\zeta \right\| )\left|c1(s)\right| ^2 ds\right\rbrace^{\frac{p}{2}} . 	
				\end{align}
				Here $C(T,p)$ is a constant dependent on $p$ and $T$.
                
				We are  now estimating $\left| a1(s)\right| ^2$, $\left| b1(s)\right| ^2$, $\left| c1(s)\right| ^2$ and $\left| d1(s)\right|^2$.
				We use the fact that the function $f_q(\cdot,\cdot)$,  $g_q(\cdot,\cdot)$ and the matrix $S^h(t,\omega)$ are bounded and satisfy the global Lipschitz condition.  For $p\geq 2$, we have
				\begin{align*}
					\left| a1(s)\right|^2& =\left( 2\left| \left\langle X_N^q(s)-X_M^q(s),  S^h(k_N(s))F_q(k_N(s), X_N^q(k_N(s)))\right.\right.\right.\nonumber\\
     &\left.\left.\left.-S^h(k_M(s))F_q(k_M(s), X_M^q(k_M(s)))\right\rangle\right|\right)^2\\
					& \leq \left(  \left\|X_N^q(s)-X_M^q(s) \right\|^2 \right.\\
					&\left. +\left\| S^h(k_N(s))F_q(k_N(s), X_N^q(k_N(s)))-S^h(k_M(s))F_q(k_M(s), X_M^q(k_M(s))) \right\|^2  
					\right) ^{\frac{2p}{p}}\\
                   & \leq \left( \left\|X_N^q(s)-X_M^q(s) \right\|^2\right.\\
                    &\left.+2\left|S^h(k_N(s))\right|_F^2\left\|F_q(k_N(s), X_N^q(k_N(s))) -F_q(k_M(s), X_M^q(k_M(s)))\right\|^2\right.\\
                    &\left. +2\left\|F_q(k_M(s), X_M^q(k_M(s))) \right\|^2\left|S^h(k_N(s))-S^h(k_M(s)) \right|_F   \right)^{\frac{2p}{p}} .
     \end{align*}
      Using Assumption \ref{assum2} (A2.2) yields 
     \begin{align*}
					\left| a1(s)\right|^2& \leq  3^{p-1}\left( \left\|X_N^q(s)-X_M^q(s) \right\|^{2p}\right.\nonumber\\
                    &\left.+2k_1\left\|F_q(k_N(s), X_N^q(k_N(s)))
					-F_q(k_M(s), X^q(k_M(s)))\right\|^{2p}\right.\\
                    &\left.+2k_1\left\|F_q(k_M(s), X_M^q(k_M(s))) \right\|^{2p}\left( \left| k_N(s)- k_M(s) \right|_F\right)  ^{2p}\right)^{\frac{2}{p}}.
                    \end{align*}
                      This means that
                     \begin{align}\label{equa47}
					\left| a1(s)\right|^2& \leq  3^{p-1}\left( \left\|X_N^q(s)-X_M^q(s) \right\|^{2p}\right.\nonumber\\
                    &\left.+2k_1\left\|F_q(k_N(s), X_N^q(k_N(s)))-F_q(k_M(s), X_M^q(k_M(s)))\right\|^{2p} \right.\nonumber\\
                    &\left.+2k_1\left\|F_q(k_M(s), X_M^q(k_M(s))) \right\|^{2p}\left[ \left( \frac{T}{N}\right)^{2p} +\left( \frac{T}{M}\right)^{2p}\right] \right)^{\frac{2}{p}}, ~s\in \left[0,T \right]. 
				\end{align}
				Note that we use the following  inequality  to obtain the last inequality \eqref{equa47}
				\begin{align*}
                 \left\| k_N(s)-s+s-  k_M(s)\right\|&\leq\left\| k_N(s)-s\right\|+ \left\| s-k_M(s)\right\|\\
                 &\leq \frac{T}{N}+\frac{T}{M}. 
                 \end{align*}
				Using Remark \ref{rem1} and Assumption \ref{assum2} yields
                \begin{align}\label{equa20a}
                \left\| F_q(k_M(s), X_M^q(k_M(s)))\right\|^{2p}&=\left\|f(k_M(s),X_M^q(k_M(s)))-J_q(k_M(s))X^q_M(k_M(s))\right\| ^{2p}\nonumber\\
                &\leq 2^{2p-1}\left\|f(k_M(s),X_M^q(k_M(s)))\right\| ^{2p}\nonumber\\
                &+\left\|J_q(k_M(s))X_M^q(K_M(s))\right\| ^{2p}\nonumber\\
                &\leq 2^{2p-1}L_q^{2p}+2^{2p-1}k_1^{2p}\left\|X_M^q(k_M(s))\right\| ^{2p}\nonumber\\
                &\leq C+C\left\|X_M^q(k_M(s))\right\| ^{2p}.
                \end{align}
In another hand by using Assumption \ref{assum2} and Remark \ref{rem1}, we have also
				\begin{align*}
					&\left\| F_q(k_N(s), X_N^q(k_N(s)))-F_q(k_M(s), X_M^q(k_M(s))) \right\|^{2p} \nonumber\\
					&~~~~~~~~~~~~~~~~~~~~~=\left\|f(k_N(s),X_N^q(k_N(s)))-J_q(k_N(s))X_N^q(k_N(s))\right.\nonumber\\
     &~~~~~~~~~~~~~~~~~~~~~\left.-f(k_M(s),X_M^q(k_M(s)))+J_q(k_M(s))X_M^q(k_M(s))\right\| ^{2p}\nonumber\\
					&~~~~~~~~~~~~~~~~~~~~~\leq 2^{2p-1}\left\|f(k_N(s),X_N^q(k_N(s)))-f(k_M(s),X_M^q(k_M(s)))\right\|^{2p} \nonumber\\
					&~~~~~~~~~~~~~~~~~~~~~+2^{2p-1}\left|J_q(k_N(s))X_N^q(k_N(s))-J_q(k_M(s))X_M^q(k_M(s)) \right| _F^{2p} \nonumber\\
					&~~~~~~~~~~~~~~~~~~~~~\leq 2^{2p-1}L_q\left\|X_N^q(k_N(s))-X_M^q(k_M(s))\right\|^{2p}\nonumber\\
					&~~~~~~~~~~~~~~~~~~~~~+C\left|J_q(k_N(s))(X_M^q(k_M(s))- X_N^q(k_N(s)))\right| _F ^{2p} \nonumber\\
                    &~~~~~~~~~~~~~~~~~~~~~+C\left|X_M^q(k_M(s))(J_q(k_M(s))-J(k_N(s))) \right| _F ^{2p}.
                    	\end{align*}
                      Using  Assumption \ref{assum2}, we obtain
                    \begin{align*}
					&\left\| F_q(k_N(s), X_N^q(k_N(s)))-F_q(k_M(s), X_M^q(k_M(s))) \right\|^{2p} \nonumber\\
					&~~~~~~~~~~~~~~~~~~~~~\leq C\left\|X_N^q(k_N(s))-X_M^q(k_M(s))\right\|^{2p}\nonumber\\ &~~~~~~~~~~~~~~~~~~~~~+C\left|J_q(k_N(s)))\right|_F^{2p}\left\|X_M^q(k_M(s))- X_N^q(k_N(s))\right\|^{2p} \nonumber\\
					&~~~~~~~~~~~~~~~~~~~~~+C\left\| X_M^q(k_M(s))\right\|^{2p} \left\|J_q(k_M(s))-J(k_N(s))\right\| _F^{2p}\nonumber\\
					&~~~~~~~~~~~~~~~~~~~~~\leq C\left\|X_N^q(k_N(s))-X_M^q(k_M(s))\right\|^{2p}\nonumber\\
                    &~~~~~~~~~~~~~~~~~~~~~+C \left\| X_M^q(k_M(s))\right\|^{2p}\left|k_M(s)-k_N(s)\right| ^{2p}\nonumber\\
					&~~~~~~~~~~~~~~~~~~~~~\leq C\left\|X_M^q(k_M(s))-X_N^q(k_N(s))\right\|^{2p}\nonumber\\
                    &~~~~~~~~~~~~~~~~~~~~~+ C\left\| X_M^q(k_M(s))\right\|^{2p}\left[\left(\frac{T}{N}\right)^{2p} +\left( \frac{T}{M}\right)^{2p}\right],~ s\in \left[0,T \right].
				\end{align*}
              Note that the constant $C$ changes from line to line. 
              
				Using the fact that 
                \begin{align*}
               \| X_N^q(k_N(s))-X_M^q(k_N(s))\|=&\|X_N^q(k_N(s))-X_N^q(s)+X_N^q(s)-X_M^q(s)\\
               &+X_M^q(s)-X_M^q(k_M(s))\|\\
                \leq & \|X_N^q(k_N(s))-X_N^q(s)\|+\|X_N^q(s)-X_M^q(s)\|\\
                &+\|X_M^q(s)-X_M^q(k_M(s))\|,
                \end{align*}
                it follows that
				\begin{align}\label{equa48}
					&\left\| F_q(k_N(s), X_N^q(k_M(s)))-F_q(k_M(s), X_N^q(k_M(s))) \right\|^{2p}\leq C\left\|X_N^q(s)-X_M^q(s))\right\|^{2p} \nonumber\\ &~~~~~~~~~~~~~~~~~~+ C\left\|X_N^q(s)-X_N^q(k_N(s))\right\|^{2p}+C\left\|X_M^q(s)-X_M^q(k_M(s))\right\|^{2p}\nonumber\\
					&~~~~~~~~~~~~~~~~~~+C\left\| X_M^q(k_M(s))\right\|^{2p}(\left( \frac{T}{N}\right)^{2p} +\left( \frac{T}{M}\right)^{2p}),
					~s\in \left[0,T \right].
				\end{align}
				Replacing \eqref{equa20a} and \eqref{equa48} in \eqref{equa47} , we obtain
				\begin{align}\label{equa49}
					\left| a1(s)\right|^2& \leq  C \left(\left\|X_N^q(s)-X_M^q(s) \right\|^{2p}+\left\|X_N^q(s)-X_N^q(k_N(s))\right\|^{2p} \right.~~~~~~~~~~~~~~~~~~~~~~~~~\nonumber\\
					& \left. +\left\|X_M^q(s)-X_M^q(k_M(s))\right\|^{2p}+\left( \frac{T}{N}\right)^{2p} +\left( \frac{T}{M}\right)^{2p}\right.\nonumber\\
					&\left.+\left\| X_M^q(k_M(s))\right\|^{2p}(\left( \frac{T}{N}\right)^{2p} +\left( \frac{T}{M}\right)^{2p})\right)^{\frac{2}{p}}, ~s\in \left[0,T \right].
				\end{align}
				Additionally, we use the same process as in the estimation of $a1(\cdot)$ and  obtain 
                the following estimations for $c1(.)$ and $d1(.)$
				\begin{align}\label{equa50}
					\left| c1(s)\right|^2&=\left( 2\left| \left\langle  X_N^q(s)-X_M^q(s),S^h(k_N(s))g_q(k_N(s),X^q(k_N(s)))\right\rangle\right.\right.~~~~~~~~~~~~~~~~~~~\nonumber\\
     &\left.\left.  - \left\langle  X_N^q(s)-X_M^q(s),S^h(k_M(s))g_q(k_M(s),X^q(k_M(s))) \right\rangle\right| \right)^2 \nonumber\\
					& \leq  C\left( \left\|X_N^q(s)-X_M^q(s) \right\|^{2p}+\left\|X_N^q(s)-X_N^q(k_N(s))\right\|^{2p} \right.\nonumber\\
					& \left. +\left\|X_M^q(s)-X_M^q(k_M(s))\right\|^{2p}+\left( \frac{T}{N}\right)^{2p} +\left( \frac{T}{M}\right)^{2p}\right)^{\frac{2}{p}},\,\,\,\,\text{ for all}\,\,s\in \left[0,T\right],
				\end{align}
				and 
				\begin{align}\label{equa51}
					&\left| d1(s)\right|^2=
                    \left( \left|S^h(k_N(s))g_q(k_N(s),X^q(k_N(s)))-S^h(k_M(s))g_q(k_M(s),X^q(k_M(s))) \right|_F^2  \right)^2 \nonumber\\
					& ~~~~~~~~~~~~\leq  C(p)\left(\left\|X_N^q(s)-X_M^q(s)\right\|^{2p}+\left\|X_N^q(s)-X_N^q(k_N(s))\right\|^{2p} \right.\nonumber\\
					& ~~~~~~~~~~~~~~\left. +\left\|X_M^q(s)-X_M^q(k_M(s))\right\|^{2p}+\left( \frac{T}{N}\right)^{2p} +\left( \frac{T}{M}\right)^{2p}\right)^{\frac{2}{p}}\,\,\, \text{ for all}\,,s\in \left[0,T\right].
				\end{align}
				We now  estimating $b1(.)$
				\begin{align}\label{equa52}
					&\left| b1(s)\right| ^2=\left|2h \left\langle X_N^q(s)-X_M^q(s),S^h(k_N(s))\frac{\partial f_q}{\partial s}(k_N(s), X_N^q(k_N(s)))\right.\right.\nonumber\\
					&\left.\left.-S^h(k_M(s))\frac{\partial f_q}{\partial s}(k_M(s), X_M^q(k_M(s)))\right\rangle\right| ^2\nonumber\\
					&\leq \left( h^2\left\| X_N^q(s)-X_M^q(s) \right\|^2\right.\nonumber\\
					&\left.+\left\|S^h(k_N(s))\frac{\partial f_q}{\partial s}(k_N(s), X_N^q(k_N(s)))-S^h(k_M(s))\frac{\partial f_q}{\partial s}(k_M(s), X_M^q(k_M(s)))\right\| ^2  \right) ^{\frac{2p}{p}}\nonumber\\
					& \leq \left( 2^{p-1}h^{2p}\left\| X_N^q(s)-X_M^q(s) \right\|^{2p}\right.\nonumber\\
					&\left.+ 2^{3p-2} \left| \frac{\partial f_q}{\partial s}(k_M(s), X_M^q(k_M(s)))\right|^{2p}\left\| S^h(k_N(s))- S^h(k_M(s)) \right\|^{2p} \right.\nonumber\\
					&\left.+ 2^{3p-2}\left| S^h(k_N(s))\right|_F^{2p}\left\|\frac{\partial f_q}{\partial s}(k_N(s), X_N^q(k_N(s))) - \frac{\partial f_q}{\partial s}(k_M(s), X_M^q(k_M(s)))\right\|  ^{2p}\right) ^{\frac{2}{p}} \nonumber\\
					& \leq \left( 2^{p-1}h^{2p}\left\| X_N^q(s)-X_M^q(s) \right\|^{2p}+ 2^{3p-2}\left[  k_2^{2p}\left\| S^h(k_N(s))- S^h(k_M(s)) \right\| ^{2p}\right.\right. \nonumber\\
					&+k_1^{2p}\left.\left.\left\|\frac{\partial f_q}{\partial s}(k_N(s), X_N^q(k_N(s))) - \frac{\partial f_q}{\partial s}(k_M(s), X^q_M(k_M(s)))\right\|  ^{2p}\right] \right) ^{\frac{2}{p}}  ,~s\in \left[0,T \right].
				\end{align}
				Using (A2.3 ) of  Assumption \ref{assum2} yields
				\begin{align*}
					\left\|\frac{\partial f_q}{\partial s}(k_N(s), X_N^q(k_N(s)))-\frac{\partial f_q}{\partial s}(k_M(s), X_M^q(k_M(s))) \right\|^{2p}&\leq k_2\left\|k_M(s)-k_N(s)\right\| ^{2p}\\
					\leq C(\left( \frac{T}{N}\right)^{2p} +\left( \frac{T}{M}\right)^{2p}),  ~s\in \left[0,T\right].
				\end{align*}
				Equation \eqref{equa52} becomes
				\begin{align*}
					\left| b1(s)\right| ^2&\leq  \left( 2^{p-1}h^{2p}\left\| X_N^q(s)-X_M^q(s) \right\|^{2p}+   C\left\| S^h(k_N(s))- S^h(k_M(s)) \right\|^{2p}\right.\nonumber\\
					&\left. 
					+C\left(\left( \frac{T}{N}\right)^{2p} +\left( \frac{T}{M}\right)^{2p}\right)\right) ^{\frac{2}{p}},~s\in \left[0,T\right].
				\end{align*}
				Using the same process as in the estimation of $a1$, we have
				\begin{align}\label{equa53}
					\left| b1(s)\right| ^2&\leq C\left(h^{2p}\left\|X_N^q(s)-X_M^q(s) \right\|^{2p}+\left( \frac{T}{N}\right)^{2p} +\left( \frac{T}{M}\right)^{2p}\right)^{\frac{2}{p}}, ~s\in \left[0,T\right].
				\end{align}
				Replacing \eqref{equa53},\eqref{equa51} \eqref{equa50} and \eqref{equa49} in  \eqref{equa46},  and using Cauchy inequality yields
				\begin{align}\label{equa54}
					\mathbb{	E}\sup_{s\leq t}\left\|	X^q_{N}(s)-X^q_{M}(s) \right\|^{2p}&\exp(-p\left\|\zeta \right\| ) ~\nonumber\\
					& \leq C(2+h^{2p})\mathbb{E}\int_{0}^{t}\exp (-p\left\|\zeta \right\| ) \left\|X_N^q(s)-X_M^q(s) \right\|^{2p}ds\nonumber\\
					&+C\mathbb{E}\int_{0}^{t}\exp (-p\left\|\zeta \right\| )\left\|X_N^q(s)-X_N^q(k_N(s))\right\|^{2p}ds\nonumber\\
					& +C\mathbb{E}\int_{0}^{t}\exp (-p\left\|\zeta \right\| )\left\|X_M^q(s)-X_M^q(k_M(s))\right\|^{2p}ds\nonumber\\
					& +C\left[\left( \frac{T}{N}\right)^{2p} +\left( \frac{T}{M}\right)^{2p}\right]\mathbb{E}\int_{0}^{t}\left\| X_M^q(k_M(s))\right\|^{2p} \nonumber\\
					&\times \exp (-p\left\|\zeta \right\| )ds+C\left( \frac{T}{N}\right)^{2p} +C\left( \frac{T}{M}\right)^{2p}.
				\end{align}		
				Note that in \eqref{equa54}, $C=C(p,T,q)$.
             Using Lemma \ref{lem2} and  Theorem \ref{theo3a} yields
				\begin{align*}
					\mathbb{	E}\sup_{s\leq t}\exp(-p\left\|\zeta \right\| ) & \left\|	X^q_{N}(s)-X^q_{M}(s) \right\|^{2p}\\
					&\leq C\mathbb{E}\int_{0}^{t}\exp (-p\left\|\zeta \right\| )\left\|X_N^q(s)-X_M^q(s) \right\|^{2p}ds\\
					& +C\left(\frac{T}{N}\right)^{p}+C\left( \frac{T}{M}\right)^{p}\\
					& +C\left( \frac{T}{N}\right)^{2p} +C\left( \frac{T}{M}\right)^{2p}, ~t\in \left[0,T \right] .
                    \end{align*}
                    Using the fact that $\left( \frac{T}{N}\right)^{2p}\leq \left( \frac{T}{N}\right)^{p}$ we have
				\begin{align}\label{equa59}
					\mathbb{	E}\sup_{s\leq t}\exp(-p\left\|\zeta \right\| ) & \left\|	X^q_{N}(s)-X^q_{M}(s) \right\|^{2p}\nonumber\\
					& \leq C\mathbb{E}\int_{0}^{t}\exp (-p\left\|\zeta \right\| ) \left\|X_N^q(s)-X_M^q(s) \right\|^{2p}ds~~~~~~~~~~~~\nonumber\\
					&+C\left( \frac{T}{N}\right)^{p} +C\left( \frac{T}{M}\right)^{p},~t\in\left[0,T \right]. 
				\end{align}	
				Using Gronwall's lemma, we obtain the following inequality
				\begin{align}\label{equa60}
					\mathbb{	E}\sup_{s\leq t}\left(\exp(-p\left\|\zeta \right\| )  \left\|	X^q_{N}(s)-X^q_{M}(s) \right\|^{2p}\right)\leq		C\left( \frac{T}{N}\right)^{p} +C\left( \frac{T}{M}\right)^{p},~t\in\left[0,T \right].
				\end{align}		

   Fix $\epsilon >0$, Markov's inequality applied  in \eqref{equa60} yields
   {\small {
    \begin{align*}
         \mathbb{P}\left(\sup_{s\leq t}\|X^q_N(s)-X^q_M(s)\|>\epsilon\right)&=\mathbb{P}\left(\sup_{s\leq t}\exp(-p\left\|\zeta \right\| )	\left\|X^q_{N}(s)-	X^q_{M}(s)\right\| ^{2p}
        >\epsilon^{2p}\exp(-p\left\|\zeta \right\|\right)\\
        &\leq \frac{\mathbb{E}\left(\sup_{s\leq t}\exp(-p\left\|\zeta \right\| )	\left\| X^q_{N}(s)-	X^q_{M}(s)\right\| ^{2p})\right)}{\epsilon^{2p}}\\
        &\leq \frac{C}{\epsilon^{2p}}\left(\frac{1}{N^{p}}+\frac{1}{M^{p}}\right).
    \end{align*}}}
    Letting $N,M \to   \infty$ yields
    $$\sup_{N,M}\mathbb{P}\left(\sup_{s\leq t}\|X^q_N(s)-X^q_M(s)\|>\epsilon\right)\to 0.$$
	The sequence $(X_N^q(t))_{N\geq 1}$ is then a Cauchy's sequence  and converges  in probability to the continuous stochastic process that we denote $Z_q(t)$, uniformly in $t\in \left[0,T \right] $.\\
    In addition, as $f_q$ continuously differentiable in $D_q$ \footnote{Indeed  this comes from $f$ as $f_q$ and $f$ agree on $D_q$.} and $g_q$ continuous in $D_q$, the following uniform convergences in probability hold for all $t\in [0,T]$ as $N\to \infty$ 
 \begin{align*}   f_q(k_N(t),X^q_N(k_N(t)))&+\frac{\partial f_q}{\partial X}(k_N(t),X^q_N(k_N(t)))[X^q_N(t)-X^q_N(k_N(t)]\\
 &+\frac{\partial f_q}{\partial t}(k_N(t),X^q_N(k_N(t)))[t-k_N(t)]\to f_q(t,Z^q(t)) \text{ and }\\
 g_q(k_N(t),X^q_N(k_N(t)))&\to g_q(t,Z^q(t)).
 \end{align*} 
 Furthermore  $Z_q(t)$ is the solution of \eqref{equa1}, where $f$ and $g$ are replaced by $f_q$ and $g_q$ respectively.
   Indeed the proof follows the same process as in \cite[Theorem 2, p. 38]{monarticle2}).
       \\
       
        This means that $X_N^q(t)$ converges  in probability to the continuous stochastic process $Z^q(t)$ solution of \eqref{equa1}, uniformly in $t\in \left[0,T \right] $ and then 
  \begin{align*}
		\mathbb{	E}	\sup_{s\leq t}\left( \exp(-p\left\|\zeta \right\| )	\left\|X^q_{N}(s)-	Z^q(s)\right\|^{2p}\right)
		&\leq \frac{C}{N^{p}} .
	\end{align*}
\end{proof}

            At this point, we have enough keys results for the proof of our main result. 
			\subsubsection{Proof Theorem \ref{theo3}} 
           \begin{proof}
				From the proof of Lemma \ref{exis} we have this relation
                 \begin{align*}
		\mathbb{	E}	\sup_{s\leq t}\left( \exp(-p\left\|\zeta \right\| )	\left\| 	X^q_{N}(s)-	Z^q(s)\right\|^{2p}\right)
		&\leq \frac{C}{N^{p}} .
	\end{align*}
     Let  $\alpha>0$,  by using Markov's inequality, for $t\in \left[0,T \right] $, we have
	\begin{eqnarray}\label{equa19}
		&&\mathbb{	P}	\left(\sup_{s\leq t} \exp(-\frac{1}{2}\left\|\zeta \right\| )	\left\| X_N^q(s)-	Z^q(s)\right\|\geq\frac{1}{N^{\alpha}}\right) \nonumber\\
		&\leq& 
        \cfrac{	\mathbb{E}	\left( \sup_{s\leq t} \exp(-\frac{1}{2}\left\|\zeta \right\| )	\left\| X_N^q(s)-	Z^q(s)\right\|\right) ^{2p}}{\left( \dfrac{1}{N^\alpha}\right) ^{2p}} \nonumber\\
        &\leq&	
		\mathbb{	E}	\left( \sup_{s\leq t} \exp(-p\left\|\zeta \right\| )	\left\| X_N^q(s)-	Z^q(s)\right\|^{2p}\right)\dfrac{N^{2p\alpha} }{1} \nonumber\\
		&\leq& C\cfrac{N^{2p\alpha}}{N^p}= \frac{C}{N^{p-2p\alpha}} \nonumber.
	\end{eqnarray}
	For ${p-2p\alpha}>1$, we need to have  $\alpha<\frac{1}{2}$. For $ \alpha \in \left(0,\frac{1}{2}\right)$, by  choosing $p>1$,  we therefore have
	\begin{eqnarray}
	    \label{equa19a}
		\sum_N	\mathbb{P}	\left( \sup_{s\leq t} \exp(-\frac{1}{2}\left\|\zeta \right\| )	\left\| X_N^q(s)-	Z^q(s)\right\|\geq\frac{1}{N^{\alpha}}\right)
		&\leq C\sum_N\frac{1}{N^{p-2p\alpha}}
		<\infty. \nonumber
	\end{eqnarray}
	 Note that  C is a constant that does not depend on $N$. 
Following \cite[pp. 210]{gyongy1998note},  for any $q\geq 1$, we obtain
	\begin{equation*}
		\sup_{t\leq \tau^q}\left| X_N(t)-Z^q(t)\right|  \leq \beta N^{-\alpha},
	\end{equation*}
    with $$\tau^q=T\land \inf \left\lbrace t\geq 0,~ Z^q(t)\notin D_q \right\rbrace .~~~~~~~$$ 

    Furthermore, following \cite[pp. 210]{gyongy1998note} we also have 
	\begin{equation}\label{equa21}
		\sup_{t\leq T}\left| X_N(t)-Z(t)\right| \leq \beta N^{-\alpha}.
	\end{equation}
    This is done by using the fact that
    $$Y(t):=\lim_{q\to \infty} Z^q(t);~ t\leq \tau ,$$  where 
    $$\tau:=\lim_{q\to \infty } \tau^q=\inf\{t\geq 0: Z(t) \notin D\}\land T.$$
    The proof is therefore completed. 
				
				
			\end{proof}	
			\section{Application and  simulation}\label{sec3}
            In this section, we give a realistic application of SDAEs in high dimensions  and check whether our theoretical result is in agreement with numerical simulation. 
            \subsection{Application of SDAEs in electrical circuits in high dimension}
           The numerical simulation of electrical circuits requires an appropriate mathematical model \cite{denk2008efficient, gunther2000cad, schein1999stochastic}. Most commonly, the Modified Nodal Analysis (MNA) approach is employed to combine physical laws such as energy or charge conservation with the characteristic equations of the network components (see \cite{gunther2000cad} and \cite[Chapter 2]{denk2008efficient}). This modeling is often performed while neglecting internal electrical noise \cite[Chapter 3]{schein1999stochastic}. In such cases, the resulting mathematical formulation takes the form of a differential-algebraic equation (DAEs)
           $$ A(t)dX(t) = f(t,X(t))dt.$$
          However, due to parasitic effects, there is an increasing number of applications in which this modeling is no longer sufficient. In such cases, noise significantly influences the system's behavior in an inherently nonlinear manner, making it impossible to treat noise merely as a small perturbation of the noise-free solution (see for example  \cite{Renate, denk2008efficient, schein1999stochastic}). As a result, time-domain simulation of noisy systems becomes essential \cite[pp. 13–20]{anile2007scientific}.
            The inner electrical noise is a random phenomenon. It is modeled as an additional stochastic current source parallel to the original electronic elements (see \cite[Figures 2.4, 2.9 and 2.11]{denk2008efficient} and \cite{demir2012analysis}). The noise intensity is given by physical characteristics \footnote{for example, we have the thermal noise of a resistor and shot noise of semiconductors}, and the noise models are added to the network equations, and the mathematical model becomes
            $$ A(t)dX(t) = f(t,X(t))dt+g(t,X(t))dW(t).$$
            Note that the unknown variable $X$ is a vector whose components are $q$: the charges of capacitances, $\phi$: the fluxes of inductances, $e$: all nodes potentials (except for the datum node i.e references nodes) $I_L$ and $I_V$: the branch currents of the current-controlled elements (inductance and voltage sources). The matrix  $A$ is the so-called capacitance matrix, and the value is the value of the capacitances and inductances in the model. The non-linear function $f$ contains the input of independent voltage and current sources as well as the conductances of the system. The function $g$ is the intensity of inner electrical noise (the quickly fluctuating stochastic perturbations of the system).\\
            In general, matrix $A$ and the functions $f$ and $g$ are very large depending on the number of compartments in the system \cite{Renate, denk2008efficient}. For example,  if we consider a system with $n_e$ nodes, $n_C$ capacitances, $n_R$ resistances, $n_L$ inductances, then we have $n_e+n_C+n_R+n_L$ unknown variables in the system that we need to solve. 
            
				\begin{figure}[h]
					\centering
					\includegraphics[width=0.8\linewidth, height=0.2\textheight]{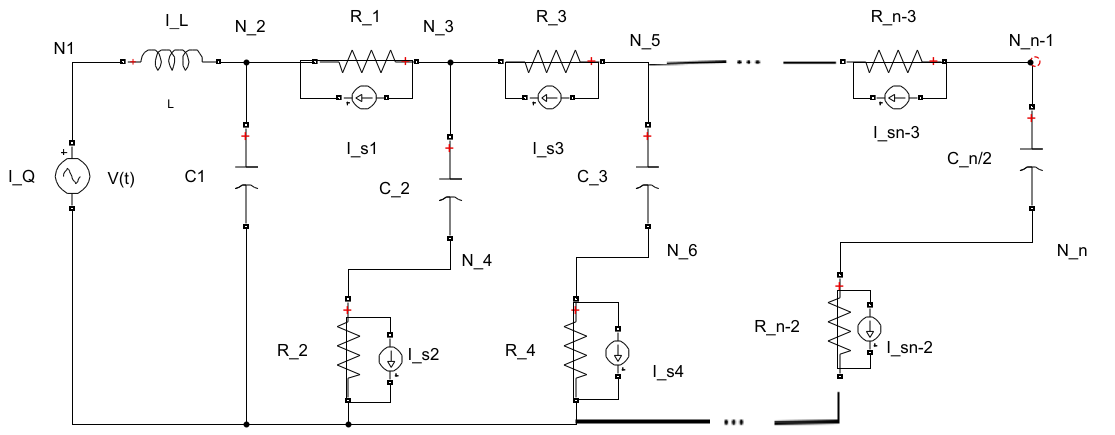}
					\caption{An example of an electrical circuit with $n$ nodes}
					\label{fig:capture}
				\end{figure}

             \subsection{Simulation in low dimension}
This example aims to verify whether our theoretical convergence order for
scheme \eqref{equa36a} as stated in Theorem \ref{theo3} is consistent with simulations. Here $D =\mathbb{R}^3$ and  we consider the following index-1 SDAEs.

	\begin{equation}\label{solu1}
		\begin{pmatrix}
			1 & 1 &0  \\
			0 & 0& 0\\
			0&-1&0
		\end{pmatrix}dX=f(t,X(t))dt+g(t,X(t))dW(t),~~ X(0)=X_0, ~~t\in\left[0,T \right],
	\end{equation}
	where
	$$f(x_1,x_2,x_3) =\begin{pmatrix}
		x_1-x_1^3+x_2-x_2^3  \\
		x_1^2+x_3\\
		x_1+x_2^3
	\end{pmatrix}$$ and $$ g(x_1,x_2,x_3)=\begin{pmatrix}
	x_1^2+x_2 & 0&x_3  \\
		0 & 0& 0\\
		0&	x_2^2&0
	\end{pmatrix}.$$
	Clearly, the matrix $A=	\begin{pmatrix}
		1 & 1 &0  \\
		0 & 0& 0\\
		0&-1&0
	\end{pmatrix}$ is a singular matrix and the functions f and g are local Lipschitz.
	By verifying the conditions of  \cite[Theorem 1]{serea2025existence}, we conclude that equation  \eqref{solu1} admits a unique solution. We can verify that the pseudo-inverse matrix $A^-$ is computed as
$$A^-=	\begin{pmatrix}
	1 & 0 &1  \\
	0 & 0& -1\\
	0&0&0
\end{pmatrix}.$$

The projector matrix $P$ is given by $$P=A^-A=	\begin{pmatrix}
	1 & 0 &0  \\
	0 & 1& 0\\
	0&0&0
\end{pmatrix},$$

and the projector matrix $R$ is therefore given by $$R=I_{3\times 3}-AA^-=	\begin{pmatrix}
	0 & 0 &0  \\
	0 & 1& 0\\
	0&0&0
\end{pmatrix}.$$

We recall that $X(t)$ admits the decomposition $X(t)=U(t)+V(t)$ and the corresponding constraint equation is given by
\begin{align*}
	A(t)V(t)+R(t)f(U(t)+V(t))&=	\begin{pmatrix}
		1 & 1 &0  \\
		0 & 0& 0\\
		0&-1&0
	\end{pmatrix}	\begin{pmatrix}
	v_1(t)  \\
	v_2(t)\\
v_3(t)
	\end{pmatrix}\nonumber\\
    &+\begin{pmatrix}
	0 & 0 &0  \\
	0 & 1& 0\\
	0&0&0
	\end{pmatrix}\begin{pmatrix}
	x_1-x_1^3+x_2-x_2^3  \\
	x_1^2+x_3\\
	x_1+x_2^3
	\end{pmatrix}\\
	&=\begin{pmatrix}
		v_1(t)+v_2(t)   \\
		(v_1(t)+u_1(t))^2+v_3(t)+u_3(t)\\
	-	v_2(t)
	\end{pmatrix}\\
	&=\begin{pmatrix}
		0   \\
		0\\
	0
	\end{pmatrix}.\\
\end{align*}
We therefore have
$$v_1(t)=0,~~v_2(t)=0 \text{ and } v_3(t)=-u_1^2(t)-u_3(t) ,~t\in [0,T].$$
 Consequentially, our constraint variable is not dependent on the noise and it is globally solvable. This means that equation \eqref{solu1} is an index-1 SDAEs.
 


 At this point we want to prove that the function $f(\cdot, \cdot)$ and $g(\cdot, \cdot)$ satisfy the monotone condition   with respect the variable $X$, i.e.
$$\left\langle (P(t)X)^T,A(t)^-f(t,X)\right\rangle +\dfrac{1}{2}\left|A^-(t)g(t,X) \right|^2_F \leq k(1+\left\|X \right\|^2 ),~t\in \left[ 0, T \right], X\in \mathbb{R}^3.$$
Indeed, we have 
\begin{align*}
\left\langle (P(t)X)^T,A(t)^-f(t,X)\right\rangle&=\begin{pmatrix}
	x_1(t)  \\
	  x_2(t)\\
   0
	\end{pmatrix}\begin{pmatrix}
	2x_1(t)-x_1^3(t)+x_2(t)   \\
	-x_1(t)-x_2^3(t)\\
	0
	\end{pmatrix}~~~~~~~~~~~~~~~~~~~~~~~~~~~~~~\\
	&=2x_1^2(t)-x_1^4(t)-x_2^4(t),
\end{align*}
and
\begin{align*}
	A^-(t)g(t,X)&=\begin{pmatrix}
		x_1^2(t)+x_2(t) & x_2^2(t) &x_3(t)  \\
		0&-x_2^2(t)&0\\
		0 & 0& 0\\
	\end{pmatrix}, \text{ this means that }\\
	|A^-(t)g(t)|^2&=(x_1^2(t)+x_2(t))^2+x_2^4(t)+x_3^2(t)+x_2^4(t)~~~~~~~~~~~~~~~~~~~~~~~~~~~~~~~~~~~~\\
	&\leq 2x_1^4(t)+2x_2^2(t)+2x_2^4(t)+2x_3^2(t).
\end{align*}
Finally we obtain
\begin{align*}
\left\langle (P(t)X)^T,A(t)^-f(t,X)\right\rangle+\dfrac{1}{2}\left|A^-(t)g(t,X)\right|^2_F& \leq 2x_1^2(t)+x_2^2(t)+x_3^2(t)\\
& \leq 2(1+x_1^2(t)+x_2^2(t)+x_3^2(t))\\
&\leq2 (1+\|X\|^2),~
t\in \left[ 0,T \right], X\in\mathbb{R}^3.
\end{align*}

 Consequently, the unique solution $X(\cdot)$ of equation \eqref{solu1}  exists and belongs to $\mathcal{M}^2(\left[0,T\right],\mathbb{R}^2 )$.\\
 

	
	
	
    For the numerical simulation we consider $N=2^{20},~T=$ and $X_0=(1,1,-1)$.\\

	
\begin{figure}[h]
	\centering
	\includegraphics[width=0.7\linewidth]{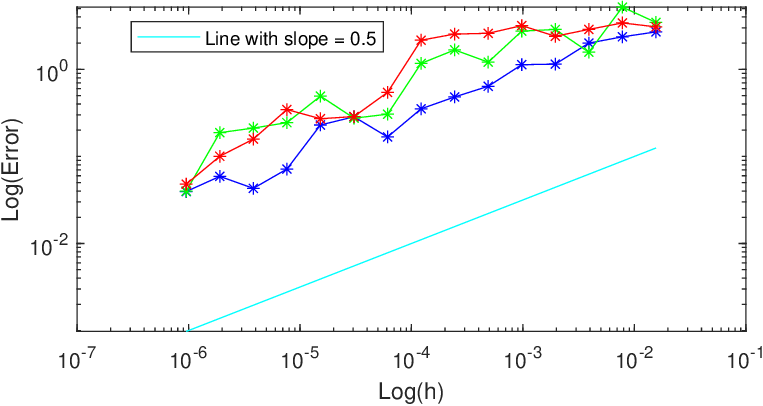}
	\caption{Pathwise convergence with our semi-implicit linearization scheme with three samples.}
	\label{fig:eroorr}
\end{figure}

In Figure \ref{fig:eroorr},
for three different realizations, we compute the convergence rate for each realization, and we obtain $0.461$, $0.4167$, and $0.4425$. Those convergence orders are close to $0.5$, this confirms our theoretical result in Theorem \ref{theo3}. Note that the error's  curves in the log scale are not straight. These are due to the fact that the bounded variable  $\epsilon$ in Theorem \ref{theo3} is random. 			

\end{document}